\documentclass{amsart}
\usepackage{colonequals}
\usepackage[alphabetic]{amsrefs}
\usepackage{longtable} 
\usepackage{adjustbox} 
\usepackage{enumerate}
\usepackage{adjustbox}
\usepackage[shortlabels]{enumitem}
\usepackage{amssymb}
\usepackage[linesnumbered,ruled,vlined]{algorithm2e}
\usepackage{fullpage}
\usepackage{xcolor}
\usepackage[colorlinks=true, pdfstartview=FitH, linkcolor=blue, citecolor=blue, urlcolor=blue]{hyperref}
\newtheorem{theorem}{Theorem}
\newtheorem{lemma}[theorem]{Lemma}

\theoremstyle{definition}

\newtheorem{question}[theorem]{Question}
\newtheorem{problem}[theorem]{Problem}
\newtheorem{conjecture}[theorem]{Conjecture}

\theoremstyle{remark}

\numberwithin{equation}{section}

\newcommand{\F}{\operatorname{\mathbb{F}}}

\allowdisplaybreaks

\makeatletter
\@namedef{subjclassname@2020}{%
  \textup{2020} Mathematics Subject Classification}
  \makeatother

\begin{document}

\title{Linear system of geometrically irreducible plane cubics over finite fields}

\author{Shamil Asgarli}
\address{Department of Mathematics and Computer Science \\ Santa Clara University \\ 500 El Camino Real \\ USA 95053}
\email{sasgarli@scu.edu}

\author{Dragos Ghioca}
\address{Department of Mathematics, University of British Columbia, Vancouver, BC V6T 1Z2}
\email{dghioca@math.ubc.ca}


\begin{abstract}
We examine the maximum dimension of a linear system of plane cubic curves whose $\mathbb{F}_q$-members are all geometrically irreducible. Computational evidence suggests that such a system has a maximum (projective) dimension of $3$. As a step towards the conjecture, we prove that there exists a $3$-dimensional linear system $\mathcal{L}$ with at most one geometrically reducible $\mathbb{F}_q$-member.
\end{abstract}

\maketitle

\section{Introduction} 

Let $\mathcal{P}$ describe a property of a degree $d$ algebraic hypersurface in $\mathbb{P}^n$. In algebraic geometry and adjacent fields, we are often interested in measuring the likelihood of the property $\mathcal{P}$ for a ``randomly chosen" hypersurface. When working over an infinite field, we can use Zariski dense open sets to show that property $\mathcal{P}$ holds generically. However, the situation over finite fields is more subtle since open sets in the relevant parameter space may not have any $\mathbb{F}_q$-points (despite being Zariski dense over $\overline{\mathbb{F}_q}$).

There are alternative methods to quantify how widespread a property $\mathcal{P}$ holds for hypersurfaces over finite fields. One method is to count the proportion of degree $d$ hypersurfaces over $\mathbb{F}_q$ satisfying $\mathcal{P}$, and consider the asymptotic proportion (either as $d\to \infty$ or $q\to\infty$). As another natural metric, we can ask for the maximum size of a linear family $\mathcal{P}$ can carry. More precisely, we can pose the following question for each finite field $\mathbb{F}_q$, and positive integers $d$ and $n$.

\begin{question}\label{quest:general}
What is the maximum value of $t\in\mathbb{N}$ for which there exist $\{F_i=0\}$ for $i=0, 1, ..., t$ such that $X_{[a_0:\cdots: a_t]} = \{a_0 F_0 + \cdot + a_t F_t = 0\}$  satisfies the property $\mathcal{P}$ for all choices $[a_0:a_1:\ldots:a_t]\in\mathbb{P}^{t}(\mathbb{F}_q)$?
\end{question}

The question can be rephrased in the language of linear systems: what is the largest (projective) dimension of a linear system $\mathcal{L}\cong \mathbb{P}^t$ of degree $d$ hypersurfaces in $\mathbb{P}^n$ such that each $\mathbb{F}_q$-member of $\mathcal{L}$ satisfies $\mathcal{P}$? An answer to Question~\ref{quest:general} measures the extent to which the property  $\mathcal{P}$ linearly propagates in the parameter space of all degree $d$ hypersurfaces in $\mathbb{P}^n$. Larger values of $t$ indicate higher levels of prevalence for $\mathcal{P}$. 

Question~\ref{quest:general} has been addressed in recent works for several specific choices of $\mathcal{P}$. For instance, $\mathcal{P}$ may correspond to the property of being smooth \cite{AGR23}, irreducible over $\mathbb{F}_q$ \cite{AGR24}, reducible over $\mathbb{F}_q$ \cite{AGR24}, or nonblocking with respect to $\mathbb{F}_q$-lines \cite{AGY23}. To illustrate some of these results, we specialize to the setting of cubic plane curves, that is, $d = 3$ and $n = 2$. Below are two concrete examples from the recent literature where Question~\ref{quest:general} has a known answer. 

\begin{theorem}[\cite{AGR23}]\label{thm:intro-smooth-cubic-plane} Let $\mathbb{F}_q$ be a finite field with characteristic $p\neq 3$. Then there is a $2$-dimensional linear system $\mathcal{L}_{\text{smo}}=\langle F_0, F_1, F_2\rangle\cong \mathbb{P}^2$ of cubic plane curves such that every $\mathbb{F}_q$-member of $\mathcal{L}_{\text{smo}}$ is smooth. Moreover, no such $3$-dimensional system exists.
\end{theorem}

\begin{theorem}[\cite{AGR24}]\label{thm:intro-irred-cubic-plane} Let $\mathbb{F}_q$ be a finite field. Then there is a $3$-dimensional linear system $\mathcal{L}_{\text{irr}}=\langle F_0, F_1, F_2, F_3\rangle\cong \mathbb{P}^3$ of cubic plane curves such that every $\mathbb{F}_q$-member of $\mathcal{L}_{\text{irr}}$ is irreducible over $\mathbb{F}_q$. Moreover, no such $4$-dimensional system exists.
\end{theorem}

Let us compare the two properties in Theorems~\ref{thm:intro-smooth-cubic-plane} and~\ref{thm:intro-irred-cubic-plane}. The smoothness condition refers to geometric smoothness, while irreducibility over $\mathbb{F}_q$ does not necessarily imply geometric irreducibility (i.e., irreducibility over $\overline{\mathbb{F}_q}$). This distinction naturally raises the goal of establishing a version of Theorem~\ref{thm:intro-irred-cubic-plane} where the conclusion is strengthened from irreducibility over $\mathbb{F}_q$ to irreducibility over $\overline{\mathbb{F}_q}$. The purpose of the present paper is to address this objective.

To summarize, we address Question~\ref{quest:general} when $\mathcal{P}$ stands for ``is geometrically irreducible" (that is, irreducible over $\overline{\mathbb{F}_q}$) for cubic plane curves: $d=3$ and $n=2$. In this special case, every linear system $\mathcal{L}$ of (projective) dimension $4$ has an $\mathbb{F}_q$-member that is a reducible plane cubic over $\mathbb{F}_q$ by Theorem~\ref{thm:intro-irred-cubic-plane}; see \cite{AGR24}*{Theorem 1.3(d)} for further details. Hence, the answer to Question~\ref{quest:general} in this setting is at most $3$. We predict that the answer is exactly $3$.

\begin{conjecture}\label{conj:main} There exists a linear system $\mathcal{L}=\langle F_0, F_1, F_2, F_3\rangle\cong\mathbb{P}^3$ of cubic plane curves where each $\mathbb{F}_q$-member of $\mathcal{L}$ is geometrically irreducible.
\end{conjecture}

As mentioned earlier, Conjecture~\ref{conj:main} strengthens Theorem~\ref{thm:intro-irred-cubic-plane}. As partial progress, we establish the following.

\begin{theorem}\label{thm:main}
There exists a linear system $\mathcal{L}=\langle F_0, F_1, F_2, F_3\rangle$ of cubic plane curves where each $\mathbb{F}_q$-member of $\mathcal{L}$ is irreducible over $\mathbb{F}_q$ and there is at most one geometrically reducible $\mathbb{F}_q$-member of $\mathcal{L}$.
\end{theorem}

To further motivate the problem and explain its arithmetic origin, we explain why the analogue of Conjecture~\ref{conj:main} fails when the base field $\mathbb{F}_q$ is replaced by an algebraically closed field $\mathbb{K}$. Since a cubic form in three variables is described by 10 coefficients, the parameter space of cubic plane curves is $\mathbb{P}^9(\mathbb{K})$. Let $\mathcal{L} = \langle F_0, F_1, F_2, F_3 \rangle$ be a linear system where each $F_i \in \mathbb{K}[x,y,z]$ is a homogeneous polynomial of degree 3. 

The set of reducible cubic curves forms a 7-dimensional subvariety $Y$ of $\mathbb{P}^9(\mathbb{K})$. This can be seen as follows: any reducible cubic polynomial can be factored as \begin{equation}\label{eq:red-cubic} (a_0 x + a_1 y + a_2 z)(b_0 x^2 + b_1 y^2 + b_2 z^2 + b_3 xy + b_4 yz + b_5 zx).
\end{equation} The variety $Y$ is the image of the natural map $\mathbb{P}^2 \times \mathbb{P}^5 \to \mathbb{P}^9$ induced by the multiplication from \eqref{eq:red-cubic}. Since $\dim(Y) + \dim(\mathcal{L}) = 7 + 3 = 10$, the intersection $C \colonequals Y \cap \mathcal{L}$ has dimension at least 1 in $\mathbb{P}^9$. Because $\mathbb{K}$ is algebraically closed, we have $C(\mathbb{K}) \neq \emptyset$, meaning that $\mathcal{L}$ has at least one reducible $\mathbb{K}$-member. In particular, Conjecture~\ref{conj:main} does \emph{not} hold over $\mathbb{K}$.

For example, if we take $\mathbb{K} = \overline{\mathbb{F}_q}$, then any 3-dimensional linear system $\mathcal{L} = \langle F_0, F_1, F_2, F_3 \rangle \cong \mathbb{P}^3$ contains a reducible $\overline{\mathbb{F}_q}$-member. The subtlety of Conjecture~\ref{conj:main} lies in the fact that, when $\mathcal{L}$ is defined over $\mathbb{F}_q$, the variety $C$, which is generically a curve, may lack $\mathbb{F}_q$-points.
 In fact, Conjecture~\ref{conj:main} is equivalent to the following statement: there exists an $\mathbb{F}_q$-linear subspace $\mathcal{L} \cong \mathbb{P}^3$ in the parameter space $\mathbb{P}^9$ such that $Y \cap \mathcal{L}$ has no $\mathbb{F}_q$-points, where $Y$ is the locus of reducible cubics. Viewed through this lens, the difficulty of Conjecture~\ref{conj:main} is tied to finding a specific ``pointless" curve inside a large-dimensional projective space. Furthermore, the above analysis also shows that the proportion of geometricaly irreducible cubic plane curves defined over $\F_q$ tends to $1$ as $q\to\infty$ (since the number of geometrically \emph{reducible} plane cubics defined over $\F_q$ is $O(q^7)$, while we have $O(q^9)$ plane cubics defined over $\F_q$).

While we focus on the case of cubic plane curves in the present paper, the same question applies to hypersurfaces of degree $d$ in $\mathbb{P}^n$ for any $d$ and $n$. 

\begin{problem}\label{open-problem-1} Determine the maximum (projective) dimension of a linear system $\mathcal{L}$ of degree $d$ hypersurfaces in $\mathbb{P}^n$ such that each $\mathbb{F}_q$-member is geometrically irreducible.
\end{problem}

By \cite{AGR24}*{Theorem 1.3(d)}, every linear system of dimension $\binom{n+d-1}{n-1}$ has an $\mathbb{F}_q$-member that is reducible over $\mathbb{F}_q$. Hence, the answer to Problem~\ref{open-problem-1} is at most $\binom{n+d-1}{n-1}-1$. On the other hand, by \cite{AGR23}*{Theorem 2}, when $p=\operatorname{char}(\mathbb{F}_q) \nmid \gcd(d, n+1)$, there exists an $n$-dimensional linear system whose $\mathbb{F}_q$-members are all smooth, hence geometrically irreducible. Hence, the answer to Problem~\ref{open-problem-1} is at least $n$. We expect the true answer to Problem~\ref{open-problem-1} to be closer to the upper bound $\binom{n+d-1}{n-1}-1$. After all, we expect most geometrically reducible hypersurfaces defined over $\mathbb{F}_q$ to be reducible over $\mathbb{F}_q$. 

We also have a related open problem with the condition ``geometrically irreducible" relaxed to ``not containing a linear component over $\overline{\mathbb{F}_q}$."

\begin{problem}\label{open-problem-2} Determine the maximum (projective) dimension of a linear system $\mathcal{L}$ of degree $d$ hypersurfaces in $\mathbb{P}^n$ such that each $\mathbb{F}_q$-member has no linear factor.
\end{problem}

By definition, the answer to Problem~\ref{open-problem-1} is less than or equal to the answer to Problem~\ref{open-problem-2}. It is reasonable to expect that the two answers agree, at least for all sufficiently large $q$ (as a function of $n$ and $d$). The heuristic is that most reducible hypersurfaces (over $\overline{\mathbb{F}_q})$ have a linear factor. Note that Conjecture~\ref{conj:main} concerns the case $n=2$ and $d=3$ for which Problems~\ref{open-problem-1} and ~\ref{open-problem-2} coincide. 

While these open problems are new, we note that the study of reducible members in a linear system of algebraic hypersurfaces is rich in literature. One case that has been investigated thoroughly is the number of reducible (or totally reducible) hypersurfaces in a pencil of hypersurfaces \cites{Lor93, Vis93, PY08}. The setting between the cited work and the present work differs in a few places. We only consider $\mathbb{F}_q$-members while the previous work is about controlling reducibility over $\overline{\mathbb{F}_q}$-members. On the other hand, we do not restrict our attention to pencils and allow large-dimensional linear systems.

\medskip 

\textbf{Structure of the paper.} We provide two proofs for Theorem~\ref{thm:main}. In Section~\ref{sec:first-proof}, we provide a non-constructive proof in the spirit of the work done in \cite{AGR24}, while in Section~\ref{sec:second-proof} we provide an explicit construction of a $3$-dimensional linear system as desired for the conclusion of Theorem~\ref{thm:main}. We believe both proofs (which are quite different in their approach) could be useful for pursuing Conjecture~\ref{conj:main}.  Appendix A provides numerical evidence (computed using SageMath) that supports Conjecture~\ref{conj:main} for all $q\leq 11$.

\medskip 

\textbf{Acknowledgments.} We thank Dino Lorenzini and the referee for valuable comments on the paper.

\section{A non-constructive proof} \label{sec:first-proof}

In this section, we discuss the construction in our previous paper \cite{AGR24} joint with Reichstein in the special case of plane cubic curves; this allows us to provide a first (non-constructive) proof of our Theorem~\ref{thm:main}. 

We begin by recalling the motivation behind the main result of \cite{AGR24} specialized to our context. Recall that the parameter space of plane conics corresponds to the projective space $\mathbb{P}^5$, as each conic can be represented by the equation
$$
b_0 x^2 + b_1 y^2 + b_2 z^2 + b_3 xy + b_4 yz + b_5 zx= 0
$$ 
for some coefficients $b_0, \ldots, b_5$. Imposing the additional condition that the conic passes through a specific point $Q_1 = [x_1 : y_1 : z_1]$ introduces one linear constraint on the coefficients $b_0, \ldots, b_5$. As a result, the space of conics passing through $Q_1$ forms a projective subspace of dimension $4$, namely $\mathbb{P}^4$. Each time we require the conic to pass through a new point $Q$, we add another linear constraint on the coefficients. However, it is not guaranteed that these constraints will be linearly independent. We say that the points $Q_1, \ldots, Q_s$ (with $s<6$) are in \emph{general position} with respect to conics if the $\mathbb{F}_q$-vector space dimension of the space of conics passing through $Q_1, \ldots, Q_s$ is exactly $6 - s$. Over the algebraic closure, the set of such $s$-tuples forms a Zariski-dense open subset. In particular, when $s = 6$, if the six points are chosen from $\mathbb{P}^2(\overline{\mathbb{F}_q})$ in general position, then no conic passes through all of them. The main contribution of \cite{AGR24} is to show that points in general position can be modeled as Galois orbits in a suitable sense. In the specific case of conics, it is possible to construct a set of six points in general position as the Galois orbit of a single point of degree 6, that is, a point with coordinates in $\mathbb{F}_{q^6}$.

More precisely, \cite{AGR24}*{Theorem 1.1} asserts the existence of a point $P\in\mathbb{P}^2(\mathbb{F}_{q^6})$ such that $P$ is not contained in any degree $2$ curve $C$ over $\mathbb{F}_q$ \cite{AGR24}*{Theorem 1.1}. Equivalently, no conic defined over $\mathbb{F}_q$ contains the Galois orbit $S=\{P, P^{\sigma}, \ldots, P^{\sigma^{5}}\}$. Here, $P^{\sigma}$ denotes the image of the point $P$ under the Frobenius map $[x:y:z]\mapsto [x^q:y^q:z^q]$. For simplicity, let us write $P_i=P^{\sigma^{i}}$ so that $S=\{P_0, ..., P_5\}$. 

Next, we follow \cite{AGR24}*{Theorem 1.3(c)} to construct a linear system of cubics $\mathcal{L}\cong \mathbb{P}^3$ where each $\mathbb{F}_q$-member of $\mathcal{L}$ is irreducible over $\mathbb{F}_q$. We will show that the same linear system $\mathcal{L}$ has at most one geometrically reducible $\mathbb{F}_q$-members, establishing a proof of Theorem~\ref{thm:main}. To construct $\mathcal{L}$, recall that the dimension of the $\mathbb{F}_q$-vector space of cubic forms in $3$ variables is $10$. Imposing the condition that a cubic passes through a specific point imposes at most $1$ linear condition on the coefficients. Since $S=\{P_0, ..., P_5\}$ has $6$ points and $S$ is defined over $\mathbb{F}_q$ (despite the fact that each $P_i$ is not individually defined over $\mathbb{F}_q$), the $\mathbb{F}_q$-vector subspace of all cubics passing through $S$ has dimension at least $10-6=4$. Let $F_0, F_1, F_2, F_3$ denote four linearly independent cubic forms in $\mathbb{F}_q[x,y,z]$ each passing through all points of $S$. Let $\mathcal{L}=\langle F_0, F_1, F_2, F_3\rangle\cong\mathbb{P}^3$ denote the $3$-dimensional linear system of cubic curves passing through $S$. 

Let $C$ be a reducible cubic curve (over $\overline{\mathbb{F}_q}$) which is an $\mathbb{F}_q$-member of $\mathcal{L}$. There are two ways in which a reducible cubic $C=L\cup Q$ can pass through the set $S$:
\begin{enumerate}[(a)]
\item Let $L_{ij}$ be the line joining $P_i$ and $P_{j}$ and $Q$ can vary in $\mathbb{P}^1$-worth of conics passing through the remaining $4$ points.
\item Let $Q_{i}$ be the conic passing through $5$ points in the set $S\setminus \{P_i\}$. Then $L$ can vary in $\mathbb{P}^1$-worth of lines passing through the remaining point $P_{i}$.
\end{enumerate}
However, if $C$ is defined over $\mathbb{F}_q$, it must be the case that $C$ is a union of three $\mathbb{F}_{q^3}$-lines, Galois conjugated by $\operatorname{Gal}(\mathbb{F}_{q^3}/\mathbb{F}_q)$; note that $C$ is assumed to be geometrically reducible, while on the other hand, by \cite{AGR24}*{Theorem 1.3(c)}, $C$ is irreducible over $\F_q$ since the entire $\F_q$-linear space $\mathcal{L}$ consists of $\F_q$-irreducible plane cubics. It is straightforward to see that \emph{exactly one} one of these curves, namely $\overline{P_0 P_3}\cup \overline{P_1 P_4} \cup \overline{P_2 P_5}$, is defined over $\mathbb{F}_q$. Hence, all $\mathbb{F}_q$-members of $\mathcal{L}$ are irreducible over $\mathbb{F}_q$ and exactly one $\mathbb{F}_q$-member of $\mathcal{L}$ fails to be geometrically irreducible. This completes the proof of Theorem~\ref{thm:main}.

\section{An explicit construction}\label{sec:second-proof}

The first proof of Theorem~\ref{thm:main} relies on the existence of a point $P\in \mathbb{P}^2(\mathbb{F}_{q^6})$ which does not lie on any conic defined over $\mathbb{F}_q$. The proof of this assertion in \cite{AGR24}*{Theorem 1.1} was obtained by an intricate counting argument and hence is nonconstructive by its nature. In this section, we offer an alternative proof of Theorem~\ref{thm:main} which has the advantage of providing an explicit construction. This latter method offers a new perspective on the problem and suggests future avenues for addressing higher-degree systems (as in Problems~\ref{open-problem-1} and \ref{open-problem-2}) through explicit methods.

We start with a lemma on reducible cubic curves containing only the monomials $x^2y, y^2z, z^2x, xyz$.  

\begin{lemma}\label{lem:main} Suppose $ax^2y + by^2z + cz^2x + dxyz=0$ is a geometrically reducible cubic curve. Then $abc=0$.
\end{lemma}

\begin{proof}
The reducible cubic has a linear factor $L$. Without loss of generality, $L=x+\beta y + \gamma z$ for some scalars $\beta, \gamma$. If $a=0$, then we are done. Hence, we may assume $a=1$ after scaling. We have:
\begin{equation}\label{eq:reducible-cubic}
x^2 y + b y^2 z + c z^2 x + d xyz =  L Q 
\end{equation}
for some quadratic factor $Q$. We match the coefficients on both sides of \eqref{eq:reducible-cubic} to prove that $b=0$. We proceed in five steps:
\begin{enumerate}
\item The cubic has no $x^3$ term, so $Q$ has no $x^2$ term. The term $x^2 y$ can only be constructed from multiplying $x$ from $L$ with a term in $xy$ from $Q$; thus, the coefficient of $xy$ in $Q$ must be $1$.
\item If $\beta\ne 0$, then $Q$
has no $y^2$ term; in that case, $LQ$ has the term $(\beta y)\cdot xy$ which leads to the term $xy^2$ in the cubic that cannot be canceled, a contradiction. Therefore, $\beta=0$.
\item The cubic has no $xy^2$ term and $\beta = 0$, so $Q$ has no $y^2$ term. The cubic has no $x^2z$ term and $Q$ has no $x^2$ term, so $Q$ has no $xz$ term. 
\item So, $Q=xy + \delta_1 yz + \delta_2 z^2$ and $L=x+\gamma z$. From \eqref{eq:reducible-cubic}, we see $\gamma \delta_1 =0$. If $\gamma = 0$, then $x$ divides the cubic, implying that $b=0$, as desired.
\item If $\gamma\neq 0$, then we have $\delta_1=0$. In this case, \eqref{eq:reducible-cubic} reads:
$$
(x + \gamma z)(xy + \delta_2 z^2) = x^2y + by^2z + cz^2x + dxyz.
$$
We obtain $b=0$, as desired. 
\end{enumerate}
Thus, any geometrically reducible cubic of the form $ax^2y + by^2z + cz^2x + dxyz=0$ satisfies $abc=0$.
\end{proof}

We will now present the second proof of our main theorem.

\begin{proof}[Proof of Theorem~\ref{thm:main}] Consider the linear system $\mathcal{L}_1 = \langle x^2 y, y^2 z, z^2 x, xyz\rangle$. By the Normal Basis Theorem, there exists an element $\alpha\in\mathbb{F}_{q^3}$ such that $\alpha, \alpha^q, \alpha^{q^2}$ forms a basis of $\mathbb{F}_{q^3}$ as an $\mathbb{F}_q$-vector space. We construct a new linear system from $\mathcal{L}_1$ where $x$, $y$, and $z$ are replaced by appropriate linear forms. Let
\begin{align*}
F &= (\alpha x + \alpha^q y + \alpha^{q^2} z)^2 (\alpha^q x + \alpha^{q^2} y + \alpha z), \\ 
G &= (\alpha^q x + \alpha^{q^2} y + \alpha z)^2 (\alpha^{q^2} x + \alpha y + \alpha^q z),  \\ 
H &= (\alpha^{q^2} x + \alpha y + \alpha^q z)^2(\alpha x + \alpha^q y + \alpha^{q^2} z), \\  
T &= (\alpha x + \alpha^q y + \alpha^{q^2} z)(\alpha^q x + \alpha^{q^2} y + \alpha z)(\alpha^{q^2} x + \alpha y + \alpha^q z).
\end{align*}
Consider the linear system $\mathcal{L}_2=\langle F, G, H, T\rangle$. The Frobenius map $t\mapsto t^{q}$ sends $F\mapsto G\mapsto H\mapsto F$ and fixes $T$. Thus, the linear system $\mathcal{L}_2$ is defined over $\mathbb{F}_q$, meaning that we can find new generators $R_0, R_1, R_2, R_3\in \mathbb{F}_q[x,y,z]$ with $\deg(R_i)=3$ such that $\mathcal{L}_2=\langle R_0, R_1, R_2, R_3\rangle$. We claim that each $\mathbb{F}_q$-member of $\mathcal{L}_2$ is geometrically irreducible except the member $T\in\mathcal{L}_2$ which is a union of three lines conjugated by $\operatorname{Gal}(\mathbb{F}_{q^3}/\mathbb{F}_q)$. Indeed, we have a new coordinate system induced by the linear transformation:
\begin{align*}
    x'&=\alpha x+\alpha^q y + \alpha^{q^2} z \\
    y'&=\alpha^q x+\alpha^{q^2} y + \alpha z \\ 
    z'&=\alpha^{q^2} x+\alpha y + \alpha^q z
\end{align*}
Applying Lemma~\ref{lem:main} in the new coordinate system, we see that any geometrically reducible $\mathbb{F}_q$-member of $\mathcal{L}_2$ given by
\begin{align*}
a F + b G + c H + d T = 0,
\end{align*}
satisfies $abc=0$. After applying the Frobenius map $t\mapsto t^q$ twice and using the fact that $T$ is defined over $\mathbb{F}_q$, we get two additional equations:
\begin{align*}
a G + b H + c F + d T &= 0, \\ 
a H + b F + c G + d T &= 0 
\end{align*}
Since $abc=0$, at least one of $a, b, c$ is zero. The three equations above and the linear independence of $F, G, H, T$ imply $a=b=c=0$. Hence, the only geometrically reducible $\mathbb{F}_q$-member of $\mathcal{L}_2$ is $\{T=0\}$. Note that $\{T=0\}$ is irreducible over $\mathbb{F}_q$. Thus, the linear system $\mathcal{L}_2$ satisfies the desired properties.
\end{proof}

\section*{Appendix A: computational evidence for the conjecture}

We verified Conjecture~\ref{conj:main} for all $q\leq 11$ using SageMath \cite{SageMath}. It suffices to randomly generate a cubic linear system $\mathcal{L}=\langle F_0, F_1, F_2, F_3\rangle$ until all $\mathbb{F}_q$-members of $\mathcal{L}$ are geometrically irreducible. The following algorithm formalizes this procedure.

\begin{algorithm}[H]
\caption{Verifying Conjecture~\ref{conj:main} for $q \leq 11$}
\SetAlgoLined
\textbf{Input:} Prime power $q$, the base field $\mathbb{F}_q$.\\
\textbf{Output:} A cubic linear system $\mathcal{L} = \langle F_0, F_1, F_2, F_3 \rangle$ such that all $\mathbb{F}_q$-members are geometrically irreducible.\\

Repeat until all $\mathbb{F}_q$-members of $\mathcal{L}$ are geometrically irreducible: \\
\quad Randomly generate coefficients $c_0, \ldots, c_9 \in \mathbb{F}_q$ to define a cubic form 
\[
F = c_0 x^3 + c_1 y^3 + c_2 z^3 + c_3 x^2 y + c_4 x y^2 + c_5 y^2 z + c_6 y z^2 + c_7 z^2 x + c_8 z x^2 + c_9 x y z.
\]
\quad Construct four independent forms $F_0, F_1, F_2, F_3$ as above.\\
\quad Define $\mathcal{L} = \langle F_0, F_1, F_2, F_3 \rangle$.\\
\quad For each $\mathbb{F}_q$-member of $\mathcal{L}$ parametrized by $\vec{a}=(a_0, a_1, a_2, a_3)$ with $a_i\in\mathbb{F}_q$, set:
\[
F_{\vec{a}} = a_0 F_0 + a_1 F_1 + a_2 F_2 + a_3 F_3 \quad \text{where} \quad a_0, a_1, a_2, a_3 \in \mathbb{F}_q:
\]
\qquad If $\{F_{\vec{a}}=0\}$ is geometrically reducible, discard $\mathcal{L}$ and return to Step 4.\\
Return $\mathcal{L}$.
\end{algorithm}

\medskip 

The following table lists the successful linear systems for $q\in \{2,3,4,5,7,8,9,11\}$. 
\medskip 

\renewcommand{\arraystretch}{1.5}
\setlength\LTleft{-1cm} 
\setlength\LTright{\fill} 
\small
\begin{longtable}{|c|c|}
 \hline
 \multicolumn{2}{|c|}{$q=2$} \\ 
 \hline
 $F_0=x^2 y+x^2 z+y^2 z$ & $F_2=x y^2+y^3+x y z+x z^2$ \\ 
 $F_1=x^3+yz^2 $ & $F_3=x^2 y+xy^2+x z^2+z^3$ \\ 
 \hline \hline 
 \multicolumn{2}{|c|}{$q=3$} \\ 
 \hline
 $F_0=y^3 + x^2z + y^2z + yz^2 + z^3$ & $F_2=x^3 - x^2 y - x y^2 + x z^2 - y z^2$ \\ 
 $F_1=x^3 - xy^2 + y^2 z - xz^2 + yz^2 - z^3$ & $F_3=-x^3 - x^2 y + y^3 + x^2 z - xz^2$ \\ 
 \hline \hline
 \multicolumn{2}{|c|}{$q=4$} \\
 \hline
 $F_0=x^2y + y^3 + x^2z + xyz + yz^2$ & $F_2=x^3 + xy^2 + y^2z + xz^2 + yz^2$ \\ 
 $F_1=x^2y + xyz + y^2 z + z^3$ & $F_3=x^3 + yz^2$ \\  
 \hline \hline 
 \multicolumn{2}{|c|}{$q=5$} \\
 \hline
 $F_0=2x^2y + xy^2 + y^3 + xz^2 + yz^2$ & $F_2=2x^3 + x^2y + xy^2 + y^3 - 2x^2z - xyz - y^2z + xz^2 + 2yz^2$ \\ 
 $F_1=x^2y + 2xy^2 - 2y^3 - 2x^2z + 2y^2z - 2xz^2 - yz^2$ & $F_3=-2x^2y - 2xy^2 - x^2z - 2xyz + y^2z - xz^2 + 2z^3$ \\ 
 \hline \hline 
 \multicolumn{2}{|c|}{$q=7$} \\
 \hline
 $F_0=-x^3 - 3xy^2 + y^3 + 3y^2z + xz^2 - 2yz^2 + 3z^3$ & $F_2=x^3 - 2x^2y + y^3 - x^2z - 3xyz - 2y^2z + xz^2 - 3z^3$ \\ 
 $F_1=3x^3 - 3x^2y - 3xy^2 - 3y^3 + xyz - 2y^2z - 2z^3$ & $F_3=-3x^3 - 2x^2y + 2xy^2 + 2y^3 - 2x^2z - 2y^2z - xz^2 + 3z^3$ \\ 
 \hline \hline 
 \multicolumn{2}{|c|}{$q=8$} \\
 \hline
 $F_0=x^2y + y^2z + xz^2 + yz^2$ & $F_2=x^3 + x^2y + y^2z + xz^2 + z^3$ \\ 
 $F_1=x^2y + xy^2 + xz^2 + z^3$ & $F_3=x^2y + y^3 + x^2z + xyz + xz^2 + yz^2 + z^3$ \\ 
 \hline \hline 
 \multicolumn{2}{|c|}{$q=9$} \\
 \hline
 $F_0=-x^3 + x^2y + y^3 + x^2z + xyz - y^2z + xz^2 - yz^2$ & $F_2=x^2y + xy^2 + x^2z + xz^2 + yz^2 + z^3$ \\ 
 $F_1=xy^2 - x^2z - xyz - y^2z - z^3$ & $F_3=xy^2 - y^3 - x^2z + y^2z - yz^2$ \\ 
 \hline \hline 
 \multicolumn{2}{|c|}{$q=11$} \\
 \hline
 $F_0=-3x^3 - 5xy^2 + 2x^2z + 4y^2z - 2xz^2 - 4z^3$ & $F_2=5x^3 + 3x^2y + y^3 - 2x^2z - 5xyz - y^2z - 5xz^2 - 3yz^2 - 4z^3$ \\ 
 $F_1=x^3 + xy^2 + 2y^3 + 3x^2z + 4xyz - y^2z - 3xz^2 + 2yz^2 - z^3$ & $F_3=2x^3 - 3x^2y + 4xy^2 + 2y^3 - 5x^2z + y^2z - 2xz^2 - yz^2 + z^3$ \\ 
 \hline 
 \end{longtable}

Interestingly, the linear system we found for $\mathbb{F}_8$ has coefficients in $\mathbb{F}_{2} = \{0, 1\}$, which means that the table entry corresponding to $q = 8$ also supports Conjecture~\ref{conj:main} for $q = 2, 4, 8$. An intriguing question arises: for how large values of $k$ can we find a linear system over $\mathbb{F}_{q}$ whose $\mathbb{F}_{q^k}$-members (not just $\mathbb{F}_q$-members) are geometrically irreducible? Such a result would provide an even stronger conclusion than Conjecture~\ref{conj:main}.

\end{document}